\documentclass[12pt, reqno]{amsart}
\usepackage{amsmath, amsthm, amscd, amsfonts, amssymb, graphicx, color}
\usepackage[bookmarksnumbered, colorlinks, plainpages]{hyperref}
\hypersetup{colorlinks=true,linkcolor=red, anchorcolor=green,
citecolor=cyan, urlcolor=red, filecolor=magenta, pdftoolbar=true}

\newtheorem{theorem}{Theorem}[section]
\newtheorem{lemma}[theorem]{Lemma}
\newtheorem{proposition}[theorem]{Proposition}
\newtheorem{corollary}[theorem]{Corollary}
\theoremstyle{definition}

\theoremstyle{remark}
\newtheorem{remark}[theorem]{Remark}
\numberwithin{equation}{section}

\begin{document}

\setcounter{page}{1}

\title[seminorm and numerical radius inequalities]
{seminorm and numerical radius inequalities of operators in semi-Hilbertian spaces}

\author[M.S. Moslehian, Q. Xu \MakeLowercase{and} A. Zamani]
{M. S. Moslehian$^1$, Q. Xu$^2$ \MakeLowercase{and} A. Zamani$^{3,*}$}

\address{$^1$Department of Pure Mathematics, Ferdowsi University of Mashhad,
Center of Excellence in Analysis on Algebraic Structures (CEAAS),
P. O. Box 1159, Mashhad 91775, Iran}
\email{moslehian@um.ac.ir, moslehian@yahoo.com}

\address{$^2$Department of Mathematics, Shanghai Normal University, Shanghai 200234, P.R. China}
\email{qingxiang\_xu@126.com}

\address{$^3$Department of Mathematics, Farhangian University, Tehran, Iran}
\email{zamani.ali85@yahoo.com (Corresponding author)}



\begin{abstract}
Let $A$ be a positive bounded operator on a Hilbert space
$\big(\mathcal{H}, \langle \cdot, \cdot\rangle \big)$.
The semi-inner product ${\langle x, y\rangle}_A := \langle Ax, y\rangle$, $x, y\in\mathcal{H},$
induces a seminorm ${\|\cdot\|}_A$ on $\mathcal{H}$.
Let ${\|T\|}_A,\ w_A(T),$ and $c_A(T)$ denote the $A$-operator seminorm,
the $A$-numerical radius, and the $A$-Crawford number of an operator $T$ in the semi-Hilbertian space
$\big(\mathcal{H}, {\|\cdot\|}_A\big)$, respectively.
In this paper, we present some seminorm inequalities and equalities for semi-Hilbertian space operators.
More precisely, we give some necessary and sufficient conditions for two
orthogonal semi-Hilbertian operators satisfy Pythagoras' equality.
In addition, we derive new upper and lower bounds for the numerical radius
of operators in semi-Hilbertian spaces. In particular, we show that
\begingroup\makeatletter\def\f@size{7}\check@mathfonts
\begin{align*}
\frac{1}{16} {\|TT^{\sharp_{A}} + T^{\sharp_{A}}T\|}^{2}_{A}
+ \frac{1}{16}c_{A}\Big(\big(T^2 + (T^{\sharp_{A}})^2\big)^2\Big)
\leq w^4_{A}(T) \leq
\frac{1}{8} {\|TT^{\sharp_{A}} + T^{\sharp_{A}}T\|}^{2}_{A} + \frac{1}{2}w^2_{A}(T^2),
\end{align*}
\endgroup
where $T^{\sharp_A}$ is a distinguished $A$-adjoint operator of $T$.
Some applications of our inequalities are also provided.
\end{abstract} \maketitle
\section{Introduction}
Let $\big(\mathcal{H}, \langle\cdot, \cdot\rangle\big)$ be a complex Hilbert space equipped with the norm $\|\cdot\|$.
If $\mathcal{M}$ is a linear subspace of $\mathcal{H}$, then $\overline{\mathcal{M}}$ stands for its closure in the
norm topology of $\mathcal{H}$.
We denote the orthogonal projection onto a closed linear subspace $\mathcal{M}$ of $\mathcal{H}$ by $P_{\mathcal{M}}$.
Let $\mathbb{B}(\mathcal{H})$ denote the $C^{\ast}$-algebra of all bounded linear operators on $\mathcal{H}$
and let $\mathbb{B}(\mathcal{H})^{+}$
be the cone of positive operators of $\mathbb{B}(\mathcal{H})$, i.e.,
\begin{align*}
\mathbb{B}(\mathcal{H})^{+}
= \big\{A \in \mathbb{B}(\mathcal{H}):\quad \langle Ax, x\rangle \geq 0 \,\, \mbox{for all} \,\, x\in\mathcal{H}\big\}.
\end{align*}
For every $T\in\mathbb{B}(\mathcal{H})$ its range is denoted by $\mathcal{R}(T)$, its null space by $\mathcal{N}(T)$, and its adjoint by $T^*$.
Any $A\in \mathbb{B}(\mathcal{H})^{+}$ defines a positive semidefinite sesquilinear form:
\begin{align*}
{\langle \cdot, \cdot\rangle}_A: \,\mathcal{H}\times \mathcal{H} \rightarrow \mathbb{C},
\quad {\langle x, y\rangle}_A = \langle Ax, y\rangle.
\end{align*}
We denote by ${\|\cdot\|}_A$ the seminorm induced by ${\langle \cdot, \cdot\rangle}_A$, that is,
${\|x\|}_A = \sqrt{{\langle x, x\rangle}_A}$
for every $x\in\mathcal{H}$.
Observe that ${\|x\|}_A = 0$ if and only if $x\in \mathcal{N}(A)$. Then ${\|\cdot\|}_A$ is a norm if and only if $A$ is one-to-one, and the seminormed
space $(\mathcal{H}, {\|\cdot\|}_A)$ is a complete space if and only if $\mathcal{R}(A)$ is closed in $\mathcal{H}$.

Throughout this paper, we assume that $A\in\mathbb{B}(\mathcal{H})$ is a positive operator.
For $T\in\mathbb{B}(\mathcal{H})$, the quantity of $A$-operator seminorm of $T$ is defined by
${\|T\|}_A = \sup\big\{{\|Tx\|}_A: \,\, {\|x\|}_A =1\big\}$.
Notice that it may happen that ${\|T\|}_A = + \infty$
for some $T\in\mathbb{B}(\mathcal{H})$. For example, let $A$ be the diagonal operator on
the Hilbert space $\ell^2$ given by $Ae_n = \frac{e_n}{n!}$,
where $\{e_n\}$ denotes the canonical basis of $\ell^2$
and consider the left shift operator $T \in \mathbb{B}(\ell^2)$.
From now on we will denote
$\mathbb{B}^{A}(\mathcal{H}) :=\big\{T\in \mathbb{B}(\mathcal{H}): \,\, {\|T\|}_A < \infty\big\}$.
It can be verified that $\mathbb{B}^{A}(\mathcal{H})$
is not a subalgebra of $\mathbb{B}(\mathcal{H})$ in general and ${\|T\|}_A = 0$ if and only if $ATA = 0$.
For $T\in \mathbb{B}(\mathcal{H})$, an operator $R\in \mathbb{B}(\mathcal{H})$
is called an $A$-adjoint operator of $T$ if for every $x, y\in \mathcal{H}$,
we have ${\langle Tx, y\rangle}_A = {\langle x, Ry\rangle}_A$, that is, $AR = T^*A$.
Generally, the existence of an $A$-adjoint operator is not guaranteed.
The set of all operators that admit $A$-adjoints is denoted by $\mathbb{B}_{A}(\mathcal{H})$.
Note that $\mathbb{B}_{A}(\mathcal{H})$ is a subalgebra of $\mathbb{B}(\mathcal{H})$, which is neither closed nor dense
in $\mathbb{B}(\mathcal{H})$. Moreover, the inclusions
$\mathbb{B}_{A}(\mathcal{H}) \subseteq \mathbb{B}^{A}(\mathcal{H}) \subseteq \mathbb{B}(\mathcal{H})$ hold
with equality if $A$ is one-to-one and has a closed range.
If $T\in\mathbb{B}_{A}(\mathcal{H})$, then the ``reduced" solution of the equation $AX = T^*A$ is a distinguished $A$-adjoint
operator of $T$, which is denoted by $T^{\sharp_A}$.
Note that, $T^{\sharp_A} = A^{\dag}T^*A$ in which $A^{\dag}$ is the Moore--Penrose inverse of $A$
and the $A$-adjoint operator $T^{\sharp_A}$ verifies
\begin{align*}
AT^{\sharp_A} = T^*A, \,\, \mathcal{R}(T^{\sharp_A}) \subseteq \overline{\mathcal{R}(A)}
\,\, \mbox{and} \,\, \mathcal{N}(T^{\sharp_A}) = \mathcal{N}(T^*A).
\end{align*}
Recall that $A^{\dag}$ is the unique linear mapping from $\mathcal{R}(A) \oplus \mathcal{R}(A)^{\perp}$
into $\mathcal{H}$ satisfying the ``Moore--Penrose equations":
\begin{align*}
AXA = A, \,\, XAX = X, \,\, XA = P_{\overline{\mathcal{R}(A)}}
\,\, \mbox{and} \,\, AX = P_{\overline{\mathcal{R}(A)}}|_{\mathcal{R}(A) \oplus \mathcal{R}(A)^{\perp}}.
\end{align*}
In general, $A^{\dag} \not \in \mathbb{B}(\mathcal{H})$.
Indeed, $A^{\dag} \in \mathbb{B}(\mathcal{H})$ if and only if $A$ has closed range; see, for example, \cite{M.K.X}.
For $T, S\in\mathbb{B}_{A}(\mathcal{H})$, it is easy to see that ${\|TS\|}_A \leq {\|T\|}_A{\|S\|}_A$
and $(TS)^{\sharp_A} = S^{\sharp_A}T^{\sharp_A}$.
Notice that if $T\in\mathbb{B}_{A}(\mathcal{H})$, then $T^{\sharp_A}\in\mathbb{B}_{A}(\mathcal{H})$,
$(T^{\sharp_A})^{\sharp_A} = P_{\overline{\mathcal{R}(A)}}TP_{\overline{\mathcal{R}(A)}}$,
$\big((T^{\sharp_A})^{\sharp_A}\big)^{\sharp_A} = T^{\sharp_A}$ and so
\begin{align*}
{\|T\|}_A = \sup\big\{|{\langle Tx, y\rangle}_A|:\,\,\, x, y\in \mathcal{H},\, {\|x\|}_A = {\|y\|}_A = 1\big\}.
\end{align*}
An operator $T\in \mathbb{B}(\mathcal{H})$ is called $A$-positive if $AT \in \mathbb{B}(\mathcal{H})^{+}$,
and we write $T {\geq}_{A} 0$. Note that if $T$ is $A$-positive, then
${\|T\|}_A = \sup\big\{{\langle Tx, x\rangle}_A:\,\,\, x\in \mathcal{H},\, {\|x\|}_A = 1\big\}$.
If, in addition, $S {\geq}_{A} T {\geq}_{A} 0$, then ${\|S\|}_A \geq {\|T\|}_A$.
An operator $T\in\mathbb{B}(\mathcal{H})$ is said to be $A$-selfadjoint if $AT$ is selfadjoint,
that is, $AT = T^*A$. Observe that if $T$ is $A$-selfadjoint, then $T\in\mathbb{B}_{A}(\mathcal{H})$.
However, it does not hold, in general, that $T = T^{\sharp_A}$.
More precisely, if $T\in\mathbb{B}_{A}(\mathcal{H})$, then $T = T^{\sharp_A}$ if and only if
$T$ is $A$-selfadjoint and $\mathcal{R}(T) \subseteq \overline{\mathcal{R}(A)}$.
Note that for $T\in\mathbb{B}_{A}(\mathcal{H})$, $T^{\sharp_A}T$ and  $TT^{\sharp_A}$ are $A$-selfadjoint and $A$-positive and so
\begin{align}\label{I.S1.0}
{\|T^{\sharp_A}T\|}_A = {\|TT^{\sharp_A}\|}_A = {\|T\|}^2_A = {\|T^{\sharp_A}\|}^2_A.
\end{align}
An operator $T\in\mathbb{B}_{A}(\mathcal{H})$ is called $A$-normal if $TT^{\sharp_A} = T^{\sharp_A}T$.
It is familiar that every selfadjoint operator is normal. However, an $A$-selfadjoint operator is
not necessarily $A$-normal. For example, consider   operators
$A = \begin{bmatrix}
1 & 1 \\
1 & 1
\end{bmatrix}$
and
$T = \begin{bmatrix}
2 & 2 \\
0 & 0
\end{bmatrix}$.
Then simple computations show that $T$ is $A$-selfadjoint and
$TT^{\sharp_A} = \begin{bmatrix}
4 & 4 \\
0 & 0
\end{bmatrix} \neq \begin{bmatrix}
2 & 2 \\
2 & 2
\end{bmatrix} = T^{\sharp_A}T$.

The $A$-numerical radius and the $A$-Crawford number of $T\in\mathbb{B}(\mathcal{H})$ are defined by
\begin{align*}
w_A(T) = \sup\Big\{\big|{\langle Tx, x\rangle}_A\big|: \,\,\, x\in \mathcal{H},\, {\|x\|}_A = 1\Big\}
\end{align*}
and
\begin{align*}
c_A(T) = \inf \big\{|{\langle Tx, x\rangle}_A|:\,\,\, x\in\mathcal{H},\,{\|x\|}_A =1\big\},
\end{align*}
respectively (see \cite{Ba.Ka.Ah, Z.3} and the references therein).
Notice that it may happen that $w_A(T) = + \infty$ for some $T\in\mathbb{B}(\mathcal{H})$.
Indeed, one can take
$A = \begin{bmatrix}
1 & 0 \\
0 & 0
\end{bmatrix}$ and
$T = \begin{bmatrix}
0 & 1 \\
1 & 0
\end{bmatrix}$.
\begin{remark}
Let $A = \begin{bmatrix}
1 & 0 \\
0 & 0
\end{bmatrix}$ and
$T = \begin{bmatrix}
1 & 0 \\
0 & 2
\end{bmatrix}$.
For $n\in \mathbb{N}$, let $A_n = \begin{bmatrix}
1 + \frac{1}{n} & 0 \\
0 & \frac{1}{n}
\end{bmatrix}$.
It is easy to see that $w_A(T) = 1$ and $w_{A_n}(T) = 2$ for every $n\in \mathbb{N}$.
Thus $\displaystyle{\lim_{n\rightarrow +\infty}}w_{A_n}(T) \neq w_A(T)$.
This example shows a nontrivial generalization from the identity operator
to a general positive semidefinite operator $A$.
\end{remark}
It has recently been shown in \cite[Theorem 2.5]{Z.3} that if $T\in\mathbb{B}_{A}(\mathcal{H})$, then
\begin{align}\label{I.S1.2}
w_A(T) = \displaystyle{\sup_{\theta \in \mathbb{R}}}
{\left\|\frac{e^{i\theta}T + (e^{i\theta}T)^{\sharp_A}}{2}\right\|}_A.
\end{align}
Further, it is known that $w_A(\cdot)$ defines a seminorm
on $\mathbb{B}_{A}(\mathcal{H})$, and that for every $T\in \mathbb{B}_{A}(\mathcal{H})$,
\begin{align}\label{I.S1.1}
\frac{1}{2}{\|T\|}_{A} \leq w_A(T)\leq {\|T\|}_{A}.
\end{align}
Moreover, it is known that if $T$ is $A$-selfadjoint (or $A$-normal), then $w_A(T) = {\|T\|}_{A}$.
For proofs and more facts about $A$-numerical radius of operators, we refer the reader to \cite{Ba.Ka.Ah, Z.3}.
Some other related topics can be found in \cite{A.K.1, A.K.2, Ar.Co.Go, B.F.O, D.1, Fo.Go, H.K.S.1, H.K.S.2, K.M.Y, Wi.Se.Su, M.S, Su}.

In Section 2, we discuss some useful seminorm inequalities and equalities for semi-Hilbertian space operators.
First, we present some refinements of the triangle inequality in $\mathbb{B}_{A}(\mathcal{H})$.
Then, for $T, S \in \mathbb{B}_{A}(\mathcal{H})$, we characterize the equality ${\|T + S\|}_{A} = {\|T\|}_{A} + {\|S\|}_{A}$.
We also give some necessary and sufficient conditions for two
orthogonal semi-Hilbertian operators to satisfy Pythagoras' equality.
In addition, we prove that ${\|T + S\|}_{A} = 2\max\big\{{\|T\|}_{A}, {\|S\|}_{A}\big\}$
if and only if $w_{A}(S^{\sharp_{A}}T) = \max\big\{{\|T\|}^2_{A}, {\|S\|}^2_{A}\big\}$.

In Section 3, we derive several $A$-numerical radius inequalities for semi-Hilbertian space operators.
In particular, we obtain some refinements on the inequalities (\ref{I.S1.1}).
Moreover, for $T\in \mathbb{B}_{A}(\mathcal{H})$, we show that
\begin{align*}
w_{A}(T^2) \leq w^2_{A}(T) \leq w_{A}(T^2) + \frac{1}{2} \min\Big\{{\|T - T^{\sharp_{A}}\|}^2_{A}, {\|T + T^{\sharp_{A}}\|}^2_{A}\Big\}.
\end{align*}
Several applications of our inequalities are also provided.
As far as we know, Theorems \ref{T.3}, \ref{T.5}, \ref{T.6}, \ref{T.9}, and \ref{T.13}
are new even in the case that the underlying operator $A$ is the identity operator.
In the case that $A$ is only positive semidefinite,
some improvements of \cite[Theorems 2.10, 2.11]{Z.3} have been made; see Theorems \ref{T.8} and \ref{T.9}.
\section{seminorm inequalities and equalities for semi-Hilbertian space operators}
We start this section with a refinement of the triangle inequality for semi-Hilbertian space operators as follows.
\begin{lemma}\label{L.2}
Let $T, S\in\mathbb{B}_{A}(\mathcal{H})$. Then
\begin{align*}
{\|T + S\|}_{A} \leq 2\int_0^1 {\big\|tT + (1 - t)S\big\|}_{A}dt \leq {\|T\|}_{A} + {\|S\|}_{A}.
\end{align*}
\end{lemma}
\begin{proof}
Let $f(t) : = {\big\|tT + (1 - t)S\big\|}_{A}$ for $t\in\mathbb{R}$.
It is easy to see that the function $f: \mathbb{R} \rightarrow \mathbb{R}$ is convex,
and so by the Hermite--Hadamard inequality (see, e.g., \cite[p. 137]{P.P.T}), we have
\begin{align*}
f\left(\frac{0+1}{2}\right) \leq \frac{1}{1-0}\int_0^1 f(t)dt \leq \frac{f(0) + f(1)}{2}.
\end{align*}
Thus
\begin{align*}
{\left\|\frac{1}{2}T + \frac{1}{2}S\right\|}_{A} \leq \int_0^1 {\big\|tT + (1 - t)S\big\|}_{A}dt
\leq \frac{{\|S\|}_{A} + {\|T\|}_{A}}{2},
\end{align*}
and hence
\begin{align*}
{\|T + S\|}_{A} \leq 2\int_0^1 {\big\|tT + (1 - t)S\big\|}_{A}dt \leq {\|T\|}_{A} + {\|S\|}_{A}.
\end{align*}
\end{proof}
\begin{remark}\label{R.2.1}
The following example shows that the inequality in Lemma \ref{L.2} is a nontrivial improvement.
Consider $A = \begin{bmatrix}
1 & 0 \\
0 & 2
\end{bmatrix}$,
$T = \begin{bmatrix}
1 & 0 \\
0 & 0
\end{bmatrix}$, and
$S = \begin{bmatrix}
0 & 0 \\
1 & 0
\end{bmatrix}$.
It is easy to see that ${\|T\|}_{A} = 1, {\|S\|}_{A} = \sqrt{2}, {\|T + S\|}_{A} = \sqrt{3}$, and
\begin{align*}
\int_0^1 {\big\|tT + (1 - t)S\big\|}_{A}dt = \int_0^1 \sqrt{3t^2 -4t +2}dt = 0.98538.
\end{align*}
Therefore,
\begin{align*}
{\|T + S\|}_{A} \simeq 1.73 < 2\int_0^1 {\big\|tT + (1 - t)S\big\|}_{A}dt\simeq 1.97 < {\|T\|}_{A} + {\|S\|}_{A} \simeq 2.41.
\end{align*}
\end{remark}
Now we apply the above result to obtain an improvement of the second inequality in (\ref{I.S1.1}).
\begin{theorem}\label{T.3}
Let $T\in\mathbb{B}_{A}(\mathcal{H})$. Then
\begin{align*}
w_{A}(T) \leq
\displaystyle{\sup_{\theta \in \mathbb{R}}} \int_0^1 {\big\|te^{i\theta}T + (1 - t)T^{\sharp_{A}}\big\|}_{A}dt
\leq {\|T\|}_{A}.
\end{align*}
\end{theorem}
\begin{proof}
Let $\theta \in \mathbb{R}$. Applying Lemma \ref{L.2} with $T:=\frac{e^{i\frac{\theta}{2}}T}{2}$
and $S:=\frac{(e^{i\frac{\theta}{2}}T)^{\sharp_{A}}}{2}$, we get
\begingroup\makeatletter\def\f@size{10}\check@mathfonts
\begin{align*}
{\left\|\frac{e^{i\frac{\theta}{2}}T}{2} + \frac{(e^{i\frac{\theta}{2}}T)^{\sharp_{A}}}{2}\right\|}_{A}
\leq 2\int_0^1 {\left\|t\frac{e^{i\frac{\theta}{2}}T}{2}
+ (1 - t)\frac{(e^{i\frac{\theta}{2}}T)^{\sharp_{A}}}{2}\right\|}_{A}dt
\leq {\left\|\frac{e^{i\frac{\theta}{2}}T}{2}\right\|}_{A}
+ {\left\|\frac{(e^{i\frac{\theta}{2}}T)^{\sharp_{A}}}{2}\right\|}_{A}.
\end{align*}
\endgroup
Since
${\big\|te^{i\frac{\theta}{2}}T + (1 - t)(e^{i\frac{\theta}{2}}T)^{\sharp_{A}}\big\|}_{A}
= {\big\|te^{i\theta}T + (1 - t)T^{\sharp_{A}}\big\|}_{A}$
and ${\|(e^{i\frac{\theta}{2}}T)^{\sharp_{A}}\|}_{A} = {\|T\|}_{A}$ by (\ref{I.S1.0}), the above double inequality gives
\begin{align}\label{I.1.T.3}
{\left\|\frac{e^{i\frac{\theta}{2}}T + (e^{i\frac{\theta}{2}}T)^{\sharp_{A}}}{2}\right\|}_{A}
\leq \int_0^1 {\big\|te^{i\theta}T + (1 - t)T^{\sharp_{A}}\big\|}_{A}dt
\leq {\|T\|}_{A}.
\end{align}
Taking the supremum over $\theta \in \mathbb{R}$ in (\ref{I.1.T.3}), we deduce that
\begin{align}\label{I.2.T.3}
\displaystyle{\sup_{\theta \in \mathbb{R}}}{\left\|\frac{e^{i\frac{\theta}{2}}T
+ (e^{i\frac{\theta}{2}}T)^{\sharp_{A}}}{2}\right\|}_{A} \leq
\displaystyle{\sup_{\theta \in \mathbb{R}}} \int_0^1 {\big\|te^{i\theta}T + (1 - t)T^{\sharp_{A}}\big\|}_{A}dt
\leq {\|T\|}_{A}.
\end{align}
Finally, by (\ref{I.S1.2}) and (\ref{I.2.T.3}), we conclude that
\begin{align*}
w_{A}(T) \leq
\displaystyle{\sup_{\theta \in \mathbb{R}}} \int_0^1 {\big\|te^{i\theta}T + (1 - t)T^{\sharp_{A}}\big\|}_{A}dt
\leq {\|T\|}_{A}.
\end{align*}
\end{proof}
In the following theorem, we give a necessary and sufficient condition for the equality
${\|T + S\|}_{A} = {\|T\|}_{A} + {\|S\|}_{A}$ in $\mathbb{B}_{A}(\mathcal{H})$.
We use some ideas of \cite[Theorem 2.1]{B.B}.
\begin{theorem}\label{T.4}
Let $T, S\in\mathbb{B}_{A}(\mathcal{H})$. Then the following conditions are equivalent:
\begin{itemize}
\item[(i)] ${\|T + S\|}_{A} = {\|T\|}_{A} + {\|S\|}_{A}$.
\item[(ii)] There exists a sequence of $A$-unit vectors $\{x_n\}$ in $\mathcal{H}$ such that
\begin{align*}
\displaystyle{\lim_{n\rightarrow +\infty}}{\langle Tx_n, Sx_n\rangle}_{A} = {\|T\|}_{A}\,{\|S\|}_{A}.
\end{align*}
\end{itemize}
\end{theorem}
\begin{proof}
(i)$\Rightarrow$(ii) Let ${\|T + S\|}_{A} = {\|T\|}_{A} + {\|S\|}_{A}$.
Then there exists a sequence of $A$-unit vectors $\{x_n\}$ in $\mathcal{H}$ such that
\begin{align}\label{I.1.T.4}
\displaystyle{\lim_{n\rightarrow +\infty}}{\|Tx_n + Sx_n\|}_{A} = {\|T + S\|}_{A}.
\end{align}
For every $n\in \mathbb{N}$, we have
\begin{align*}
{\|Tx_n + Sx_n\|}_{A}\leq {\|Tx_n\|}_{A} + {\|Sx_n\|}_{A} \leq {\|Tx_n\|}_{A} + {\|S\|}_{A} \leq {\|T\|}_{A} + {\|S\|}_{A},
\end{align*}
whence
\begin{align*}
\displaystyle{\lim_{n\rightarrow +\infty}}\big({\|Tx_n\|}_{A} + {\|S\|}_{A}\big) = {\|T\|}_{A} + {\|S\|}_{A}.
\end{align*}
From this, we conclude that
\begin{align}\label{I.2.T.4.1}
\displaystyle{\lim_{n\rightarrow +\infty}}{\|Tx_n\|}_{A} = {\|T\|}_{A}.
\end{align}
Similarly, we obtain
\begin{align}\label{I.2.T.4.2}
\displaystyle{\lim_{n\rightarrow +\infty}}{\|Sx_n\|}_{A} = {\|S\|}_{A}.
\end{align}
Since
\begin{align*}
{\|Tx_n + Sx_n\|}^2_{A} = {\|Tx_n\|}^2_{A} + 2\mbox{Re}{\langle Tx_n, Sx_n\rangle}_{A} + {\|Sx_n\|}^2_{A}
\end{align*}
for every $n\in \mathbb{N}$, from (\ref{I.1.T.4}), (\ref{I.2.T.4.1}), and (\ref{I.2.T.4.2}), we obtain
\begin{align}\label{I.3.T.4}
\displaystyle{\lim_{n\rightarrow +\infty}}\mbox{Re}{\langle Tx_n, Sx_n\rangle}_{A} = {\|T\|}_{A}\,{\|S\|}_{A}.
\end{align}
In addition, for every $n\in \mathbb{N}$, we have
\begingroup\makeatletter\def\f@size{10}\check@mathfonts
\begin{align*}
\mbox{Re}^2{\langle Tx_n, Sx_n\rangle}_{A} + \mbox{Im}^2{\langle Tx_n, Sx_n\rangle}_{A}
= |{\langle Tx_n, Sx_n\rangle}_{A}|^2 \leq {\|T\|}_{A}\,{\|S\|}_{A},
\end{align*}
\endgroup
and so by (\ref{I.3.T.4}), we conclude that
$\displaystyle{\lim_{n\rightarrow +\infty}}\mbox{Im}{\langle Tx_n, Sx_n\rangle}_{A} = 0$.
It follows from (\ref{I.3.T.4}) that
\begin{align*}
\displaystyle{\lim_{n\rightarrow +\infty}}{\langle Tx_n, Sx_n\rangle}_{A} = {\|T\|}_{A}\,{\|S\|}_{A}.
\end{align*}
(ii)$\Rightarrow$(i) Suppose that there exists a sequence of $A$-unit vectors $\{x_n\}$ in $\mathcal{H}$ such that
\begin{align*}
\displaystyle{\lim_{n\rightarrow +\infty}}{\langle Tx_n, Sx_n\rangle}_{A} = {\|T\|}_{A}\,{\|S\|}_{A}.
\end{align*}
Hence $\displaystyle{\lim_{n\rightarrow +\infty}}\mbox{Re}{\langle Tx_n, Sx_n\rangle}_{A} = {\|T\|}_{A}\,{\|S\|}_{A}$.
Since
\begin{align*}
{\|T\|}^2_{A} + 2|{\langle Tx_n, Sx_n\rangle}_{A}| + {\|S\|}^2_{A}
& = {\|T\|}^2_{A} + 2|{\langle S^{\sharp_{A}}Tx_n, x_n\rangle}_{A}| + {\|S\|}^2_{A}
\\& \leq {\|T\|}^2_{A} + 2{\|S^{\sharp_{A}}Tx_n\|}_{A} + {\|S\|}^2_{A}
\\& \leq {\|T\|}^2_{A} + 2{\|S^{\sharp_{A}}\|}_{A}{\|Tx_n\|}_{A} + {\|S\|}^2_{A}
\\& \leq {\|T\|}^2_{A} + 2{\|S\|}_{A}{\|T\|}_{A} + {\|S\|}^2_{A}
= ({\|T\|}_{A} + {\|S\|}_{A})^2
\end{align*}
for every $n\in \mathbb{N}$, we have $\displaystyle{\lim_{n\rightarrow +\infty}}{\|Tx_n\|}_{A} = {\|T\|}_{A}$ and
\begin{align}\label{I.3.T.3.9}
\displaystyle{\lim_{n\rightarrow +\infty}}{\|S^{\sharp_{A}}Tx_n\|}_{A} = {\|S\|}_{A}\,{\|T\|}_{A}.
\end{align}
By a similar argument, we get
$\displaystyle{\lim_{n\rightarrow +\infty}}{\|Sx_n\|}_{A} = {\|S\|}_{A}$.
Thus,
\begin{align*}
({\|T\|}_{A} + {\|S\|}_{A})^2 &= \displaystyle{\lim_{n\rightarrow +\infty}}{\|Tx_n\|}^2_{A}
+ 2\displaystyle{\lim_{n\rightarrow +\infty}}\mbox{Re}{\langle Tx_n, Sx_n\rangle}_{A}
+ \displaystyle{\lim_{n\rightarrow +\infty}}{\|Sx_n\|}^2_{A}
\\& = \displaystyle{\lim_{n\rightarrow +\infty}}{\|(T + S)x_n\|}^2_{A}
\leq {\|T + S\|}^2_{A}
\leq ({\|T\|}_{A} + {\|S\|}_{A})^2.
\end{align*}
Hence ${\|T + S\|}_{A} = {\|T\|}_{A} + {\|S\|}_{A}$.
\end{proof}
As a consequence of Theorem \ref{T.4}, we have the following result.
\begin{corollary}\label{C.4.5}
Let $T, S\in\mathbb{B}_{A}(\mathcal{H})$ such that $S^{\sharp_{A}}T$
is $A$-positive. Then the following conditions are equivalent:
\begin{itemize}
\item[(i)] ${\|T + S\|}_{A} = {\|T\|}_{A} + {\|S\|}_{A}$.
\item[(ii)] ${\|S^{\sharp_{A}}T\|}_{A} = {\|S\|}_{A}\,{\|T\|}_{A}$.
\end{itemize}
\end{corollary}
\begin{proof}
Let ${\|T + S\|}_{A} = {\|T\|}_{A} + {\|S\|}_{A}$.
There exists a sequence of $A$-unit vectors $\{x_n\}$ in $\mathcal{H}$ such that (\ref{I.3.T.3.9}) is satisfied.
Therefore,
\begin{align*}
{\|S\|}_{A}\,{\|T\|}_{A}
= \displaystyle{\lim_{n\rightarrow +\infty}}{\|S^{\sharp_{A}}Tx_n\|}_{A}
\leq {\|S^{\sharp_{A}}T\|}_{A} \leq {\|S^{\sharp_{A}}\|}_{A}\,{\|T\|}_{A} = {\|S\|}_{A}\,{\|T\|}_{A},
\end{align*}
and hence ${\|S^{\sharp_{A}}T\|}_{A} = {\|S\|}_{A}\,{\|T\|}_{A}$.

Conversely, assume that ${\|S^{\sharp_{A}}T\|}_{A} = {\|S\|}_{A}\,{\|T\|}_{A}$.
Since $S^{\sharp_{A}}T$ is $A$-positive,   there exists a sequence of $A$-unit vectors $\{x_n\}$ in $\mathcal{H}$ such that
\begin{align*}
\displaystyle{\lim_{n\rightarrow +\infty}}{\langle S^{\sharp_{A}}Tx_n, x_n\rangle}_{A}
= {\|S^{\sharp_{A}}T\|}_{A}.
\end{align*}
Thus $\displaystyle{\lim_{n\rightarrow +\infty}}{\langle Tx_n, Sx_n\rangle}_{A} = {\|T\|}_{A}\,{\|S\|}_{A}$.
Hence, by Theorem \ref{T.4}, we get ${\|T + S\|}_{A} = {\|T\|}_{A} + {\|S\|}_{A}$.
\end{proof}
Next, we present another improvement of the triangle inequality for semi-Hilbertian space operators.
\begin{proposition}\label{P.4.9}
Let $T, S\in\mathbb{B}_{A}(\mathcal{H})$. Then
\begin{align*}
{\|T + S\|}_{A} \leq \Big({\|T^{\sharp_{A}}T + S^{\sharp_{A}}S\|}_{A}
+ 2w_{A}(S^{\sharp_{A}}T)\Big)^{1/2} \leq {\|T\|}_{A} + {\|S\|}_{A}.
\end{align*}
\end{proposition}
\begin{proof}
Let $x \in \mathcal{H}$ with ${\|x\|}_{A} = 1$. Then
\begin{align*}
{\|Tx + Sx\|}^2_{A} &= {\langle Tx, Tx\rangle}_{A} + {\langle Sx, Sx\rangle}_{A} + 2\mbox{Re}{\langle Tx, Sx\rangle}_{A}
\\& \leq {\langle T^{\sharp_{A}}Tx, x\rangle}_{A} + {\langle S^{\sharp_{A}}Sx, x\rangle}_{A} + 2\big|{\langle Tx, Sx\rangle}_{A}\big|
\\& = \big{\langle (T^{\sharp_{A}}T + S^{\sharp_{A}}S)x, x \big\rangle}_{A} + 2\big|{\langle S^{\sharp_{A}}Tx, x\rangle}_{A}\big|
\\& \leq {\|T^{\sharp_{A}}T + S^{\sharp_{A}}S\|}_{A} + 2w_{A}(S^{\sharp_{A}}T).
\end{align*}
Thus
\begin{align*}
{\|Tx + Sx\|}^2_{A} \leq {\|T^{\sharp_{A}}T + S^{\sharp_{A}}S\|}_{A} + 2w_{A}(S^{\sharp_{A}}T).
\end{align*}
Taking the supremum over unit vectors $x \in \mathcal{H}$ in the above inequality, we arrive at
\begin{align}\label{I.1.T.1}
{\|T + S\|}^2_{A} \leq {\|T^{\sharp_{A}}T + S^{\sharp_{A}}S\|}_{A} + 2w_{A}(S^{\sharp_{A}}T).
\end{align}
Further, by (\ref{I.S1.0}) and (\ref{I.S1.1}), we have
\begin{align}\label{I.2.T.1}
w_{A}(S^{\sharp_{A}}T) \leq {\|S^{\sharp_{A}}T\|}_{A}
\leq {\|S^{\sharp_{A}}\|}_{A}\,{\|T\|}_{A} = {\|S\|}_{A}\,{\|T\|}_{A}.
\end{align}
So, by (\ref{I.S1.0}) and (\ref{I.2.T.1}), we obtain
\begin{align*}
{\|T^{\sharp_{A}}T + S^{\sharp_{A}}S\|}_{A} &+ 2w_{A}(S^{\sharp_{A}}T)
\\& \leq {\|T^{\sharp_{A}}T\|}_{A} + {\|S^{\sharp_{A}}S\|}_{A} + 2{\|S\|}_{A}\,{\|T\|}_{A}
\\& = {\|T\|}^2_{A} + {\|S\|}^2_{A} + 2{\|S\|}_{A}\,{\|T\|}_{A} = \big({\|T\|}_{A} + {\|S\|}_{A}\big)^2.
\end{align*}
Hence
\begin{align}\label{I.3.T.1}
{\|T^{\sharp_{A}}T + S^{\sharp_{A}}S\|}_{A} + 2w_{A}(S^{\sharp_{A}}T)
\leq \big({\|T\|}_{A} + {\|S\|}_{A}\big)^2.
\end{align}
Utilizing (\ref{I.1.T.1}) and (\ref{I.3.T.1}), we deduce the desired result.
\end{proof}
If $T, S\in\mathbb{B}_{A}(\mathcal{H})$, then
${\|T + S\|}_{A} \leq {\|T\|}_{A} + {\|S\|}_{A} $
and hence
\begin{align}\label{I.0.T.5}
{\|T + S\|}_{A} \leq 2\max\big\{{\|T\|}_{A}, {\|S\|}_{A}\big\}.
\end{align}
In the following theorem, we mimic \cite[Theorem 2.5]{H.K.S.2} to prove a condition for the equality in (\ref{I.0.T.5}).
\begin{theorem}\label{T.5}
Let $T, S\in\mathbb{B}_{A}(\mathcal{H})$. Then the following conditions are equivalent:
\begin{itemize}
\item[(i)] ${\|T + S\|}_{A} = 2\max\big\{{\|T\|}_{A}, {\|S\|}_{A}\big\}$.
\item[(ii)] $w_{A}(S^{\sharp_{A}}T) = \max\big\{{\|T\|}^2_{A}, {\|S\|}^2_{A}\big\}$.
\end{itemize}
\end{theorem}
\begin{proof}
(i)$\Rightarrow$(ii) Let ${\|T + S\|}_{A} = 2\max\big\{{\|T\|}_{A}, {\|S\|}_{A}\big\}$. Since
\begin{align*}
{\|T + S\|}_{A} \leq {\|T\|}_{A} + {\|S\|}_{A} \leq 2\max\big\{{\|T\|}_{A}, {\|S\|}_{A}\big\},
\end{align*}
we get
\begin{align*}
{\|T\|}_{A} + {\|S\|}_{A} = 2\max\big\{{\|T\|}_{A}, {\|S\|}_{A}\big\}.
\end{align*}
Hence ${\|T\|}_{A} = {\|S\|}_{A} = \max\big\{{\|T\|}_{A}, {\|S\|}_{A}\big\}$.
Thus ${\|T + S\|}_{A} = {\|T\|}_{A} + {\|S\|}_{A}$.
By Theorem \ref{T.4}, there exists a sequence of $A$-unit vectors $\{x_n\}$ in $\mathcal{H}$ such that
\begin{align*}
\displaystyle{\lim_{n\rightarrow +\infty}}{\langle Tx_n, Sx_n\rangle}_{A} = {\|T\|}_{A}\,{\|S\|}_{A}.
\end{align*}
This implies
\begin{align}\label{I.1.T.5}
\max\big\{{\|T\|}^2_{A}, {\|S\|}^2_{A}\big\} = {\|T\|}_{A}\,{\|S\|}_{A} \leq w_{A}(S^{\sharp_{A}}T).
\end{align}
Further, by the second inequality in (\ref{I.S1.1}) and the arithmetic-geometric mean inequality, we have
\begingroup\makeatletter\def\f@size{10}\check@mathfonts
\begin{align}\label{I.2.T.5}
w_{A}(S^{\sharp_{A}}T) \leq {\|S^{\sharp_{A}}T\|}_{A} \leq {\|S^{\sharp_{A}}\|}_{A}\,{\|T\|}_{A}
= {\|T\|}_{A}\,{\|S\|}_{A} \leq \frac{{\|T\|}^2_{A} + {\|S\|}^2_{A}}{2}
\leq \max\big\{{\|T\|}^2_{A}, {\|S\|}^2_{A}\big\}
\end{align}
\endgroup
and hence
\begin{align}\label{I.3.T.5}
w_{A}(S^{\sharp_{A}}T) \leq \max\big\{{\|T\|}^2_{A}, {\|S\|}^2_{A}\big\}.
\end{align}
By (\ref{I.1.T.5}) and (\ref{I.3.T.5}), we conclude that $w_{A}(S^{\sharp_{A}}T) = \max\big\{{\|T\|}^2_{A}, {\|S\|}^2_{A}\big\}$.

(ii)$\Rightarrow$(i) Let $w_{A}(S^{\sharp_{A}}T) = \max\big\{{\|T\|}^2_{A}, {\|S\|}^2_{A}\big\}$.
From (\ref{I.2.T.5}) it follows that
\begin{align*}
\max\big\{{\|T\|}^2_{A}, {\|S\|}^2_{A}\big\} = w_{A}(S^{\sharp_{A}}T)
\leq {\|T\|}_{A}\,{\|S\|}_{A} \leq \max\big\{{\|T\|}^2_{A}, {\|S\|}^2_{A}\big\},
\end{align*}
which yields ${\|T\|}_{A} = {\|S\|}_{A}$ and $w_{A}(S^{\sharp_{A}}T) = {\|T\|}_{A}\,{\|S\|}_{A}$.
So there exists a sequence of $A$-unit vectors $\{x_n\}$ in $\mathcal{H}$ such that
$\displaystyle{\lim_{n\rightarrow +\infty}}{\langle S^{\sharp_{A}}Tx_n, x_n\rangle}_{A} = {\|T\|}_{A}\,{\|S\|}_{A}$,
or equivalently,
$\displaystyle{\lim_{n\rightarrow +\infty}}{\langle Tx_n, Sx_n\rangle}_{A} = {\|T\|}_{A}\,{\|S\|}_{A}$.
Now, by Theorem \ref{T.4}, we obtain
${\|T + S\|}_{A} = {\|T\|}_{A} + {\|S\|}_{A}$ and so ${\|T + S\|}_{A} = 2\max\big\{{\|T\|}_{A}, {\|S\|}_{A}\big\}$.
\end{proof}
It is easy to see that Pythagoras' equality does not hold for semi-Hilbertian space operators.
The following theorem characterizes when Pythagoras' equality holds for semi-Hilbertian space operators.
\begin{theorem}\label{T.6}
Let $T, S\in\mathbb{B}_{A}(\mathcal{H})$   such that $S^{\sharp_{A}}T = 0$.
Then the following conditions are equivalent:
\begin{itemize}
\item[(i)] ${\|T + S\|}^2_{A} = {\|T\|}^2_{A} + {\|S\|}^2_{A}$.
\item[(ii)] There exists a sequence of $A$-unit vectors $\{x_n\}$ in $\mathcal{H}$ such that
\begin{align*}
\displaystyle{\lim_{n\rightarrow +\infty}}{\langle T^{\sharp_{A}}Tx_n, S^{\sharp_{A}}Sx_n\rangle}_{A}
= {\|T\|}^2_{A}\,{\|S\|}^2_{A}.
\end{align*}
\end{itemize}
\end{theorem}
\begin{proof}
Due to $S^{\sharp_{A}}T = 0$, we have $T^{\sharp_{A}}(S^{\sharp_{A}})^{\sharp_{A}} = 0$.
Hence $S^{\sharp_{A}} (T^{\sharp_{A}})^{\sharp_{A}}= 0$.
Thus
\begin{align*}
{\|T + S\|}^2_{A}& = {\Big\|(T + S)^{\sharp_{A}}\big((T + S)^{\sharp_{A}}\big)^{\sharp_{A}}\Big\|}_{A}
\\& = {\Big\|T^{\sharp_{A}}(T^{\sharp_{A}})^{\sharp_{A}} + T^{\sharp_{A}}(S^{\sharp_{A}})^{\sharp_{A}}
+ S^{\sharp_{A}}(T^{\sharp_{A}})^{\sharp_{A}} + S^{\sharp_{A}}(S^{\sharp_{A}})^{\sharp_{A}}\Big\|}_{A}
\\& = {\|T^{\sharp_{A}}(T^{\sharp_{A}})^{\sharp_{A}} + S^{\sharp_{A}}(S^{\sharp_{A}})^{\sharp_{A}}\|}_{A}
\\& = {\|\big(T^{\sharp_{A}}T + S^{\sharp_{A}}S\big)^{\sharp_{A}}\|}_{A}
= {\|T^{\sharp_{A}}T + S^{\sharp_{A}}S\|}_{A}.
\end{align*}
Hence
\begin{align}\label{I.1.T.6}
{\|T + S\|}^2_{A} = {\|T^{\sharp_{A}}T + S^{\sharp_{A}}S\|}_{A}.
\end{align}

Now, let ${\|T + S\|}^2_{A} = {\|T\|}^2_{A} + {\|S\|}^2_{A}$.
Then, by (\ref{I.S1.0}) and (\ref{I.1.T.6}), it follows that
\begin{align*}
{\|T^{\sharp_{A}}T\|}_{A} + {\|S^{\sharp_{A}}S\|}_{A} = {\|T + S\|}^2_{A}
= {\|T^{\sharp_{A}}T + S^{\sharp_{A}}S\|}_{A}.
\end{align*}
Hence
\begin{align*}
{\|T^{\sharp_{A}}T\|}_{A} + {\|S^{\sharp_{A}}S\|}_{A} = {\|T^{\sharp_{A}}T + S^{\sharp_{A}}S\|}_{A}.
\end{align*}
In view of Theorem \ref{T.4}, there exists a sequence of $A$-unit vectors $\{x_n\}$ in $\mathcal{H}$ such that
\begin{align*}
\displaystyle{\lim_{n\rightarrow +\infty}}{\langle T^{\sharp_{A}}Tx_n, S^{\sharp_{A}}Sx_n\rangle}_{A}
= {\|T^{\sharp_{A}}T\|}_{A}\,{\|S^{\sharp_{A}}S\|}_{A},
\end{align*}
and so
\begin{align*}
\displaystyle{\lim_{n\rightarrow +\infty}}{\langle T^{\sharp_{A}}Tx_n, S^{\sharp_{A}}Sx_n\rangle}_{A}
= {\|T\|}^2_{A}\,{\|S\|}^2_{A}.
\end{align*}
To prove the converse, suppose that there exists a sequence of $A$-unit vectors $\{x_n\}$ in $\mathcal{H}$ such that
$\displaystyle{\lim_{n\rightarrow +\infty}}{\langle T^{\sharp_{A}}Tx_n, S^{\sharp_{A}}Sx_n\rangle}_{A}
= {\|T\|}^2_{A}\,{\|S\|}^2_{A}.$
Therefore,
\begin{align*}
\displaystyle{\lim_{n\rightarrow +\infty}}{\langle T^{\sharp_{A}}Tx_n, S^{\sharp_{A}}Sx_n\rangle}_{A}
= {\|T^{\sharp_{A}}T\|}_{A}\,{\|S^{\sharp_{A}}S\|}_{A}.
\end{align*}
It follows from Theorem \ref{T.4} that
${\|T^{\sharp_{A}}T + S^{\sharp_{A}}S\|}_{A} = {\|T^{\sharp_{A}}T\|}_{A} + {\|S^{\sharp_{A}}S\|}_{A}$.
Now, (\ref{I.S1.0}) and (\ref{I.1.T.6}) yield that ${\|T + S\|}^2_{A} = {\|T\|}^2_{A} + {\|S\|}^2_{A}$.
\end{proof}
\section{Further refinements of $A$-numerical radius inequalities for semi-Hilbertian space operators}
In this section, inspired by the numerical radius inequalities of bounded linear operators
in \cite{A.K.1, D.1, H.K.S.2}, we derive several $A$-numerical radius inequalities for semi-Hilbertian space operators.
Our first result reads as follows.
\begin{proposition}\label{P.7}
Let $T\in\mathbb{B}_{A}(\mathcal{H})$. Then
\begingroup\makeatletter\def\f@size{10}\check@mathfonts
\begin{align*}
\frac{1}{2} \max\Big\{{\|T - T^{\sharp_{A}}\|}_{A}, {\|T + T^{\sharp_{A}}\|}_{A}\Big\}
\leq w_{A}(T) \leq
\frac{1}{2} \Big({\|T - T^{\sharp_{A}}\|}^2_{A} + {\|T + T^{\sharp_{A}}\|}^2_{A}\Big)^{1/2}.
\end{align*}
\endgroup
\end{proposition}
\begin{proof}
Employing $\big((T^{\sharp_{A}})^{\sharp_{A}}\big)^{\sharp_{A}} = T^{\sharp_{A}}$,
one can easily observe that the operators
$T^{\sharp_{A}} \pm (T^{\sharp_{A}})^{\sharp_{A}}$ are $A$-normal. Thus
\begin{align}\label{I.1.T.7}
w_{A}\big(T^{\sharp_{A}} \pm (T^{\sharp_{A}})^{\sharp_{A}}\big)
= {\big\|T^{\sharp_{A}} \pm (T^{\sharp_{A}})^{\sharp_{A}}\big\|}_{A}.
\end{align}
Let $x \in \mathcal{H}$ with ${\|x\|}_{A} = 1$.
By the triangle inequality, we have
\begin{align*}
|{\langle Tx, x\rangle}_{A}|^2
&= \frac{1}{2}\Big(|{\langle T^{\sharp_{A}}x, x\rangle}_{A}|^2
+ |{\langle (T^{\sharp_{A}})^{\sharp_{A}}x, x\rangle}_{A}|^2\Big)
\\& \geq \frac{1}{4}\Big(|{\langle T^{\sharp_{A}}x, x\rangle}_{A}|
+ |{\langle (T^{\sharp_{A}})^{\sharp_{A}}x, x\rangle}_{A}|\Big)^2
\\& \geq \frac{1}{4}\Big|{\big\langle (T^{\sharp_{A}} \pm (T^{\sharp_{A}})^{\sharp_{A}})x, x\big\rangle}_{A}\Big|^2.
\end{align*}
Taking the supremum over unit vectors $x \in \mathcal{H}$, we obtain
\begin{align*}
w^2_{A}(T) \geq \frac{1}{4}w_{A}\big(T^{\sharp_{A}} \pm (T^{\sharp_{A}})^{\sharp_{A}}\big).
\end{align*}
This together with (\ref{I.1.T.7}) and (\ref{I.S1.0}) gives
\begin{align}\label{I.2.9.T.7}
w^2_{A}(T) \geq \frac{1}{4}{\big\|T^{\sharp_{A}} \pm (T^{\sharp_{A}})^{\sharp_{A}}\big\|}^2_{A}
= \frac{1}{4}{\|T \pm T^{\sharp_{A}}\|}^2_{A},
\end{align}
which yields
\begin{align}\label{I.2.T.7}
\frac{1}{2} \max\Big\{{\|T - T^{\sharp_{A}}\|}_{A}, {\|T + T^{\sharp_{A}}\|}_{A}\Big\} \leq w_{A}(T).
\end{align}
Again, let $x \in \mathcal{H}$ with ${\|x\|}_{A} = 1$. By the parallelogram identity, we have
\begin{align*}
|{\langle Tx, x\rangle}_{A}|^2
&= \frac{1}{2}\Big(|{\langle Tx, x\rangle}_{A}|^2 + |{\langle T^{\sharp_{A}}x, x\rangle}_{A}|^2\Big)
\\& = \frac{1}{4}\Big(\big|{\big\langle (T + T^{\sharp_{A}})x, x\big\rangle}_{A}\big|^2
+ \big|{\big\langle (T - T^{\sharp_{A}})x, x\big\rangle}_{A}\big|^2\Big)
\\& \leq \frac{1}{4}\Big({\|T - T^{\sharp_{A}}\|}^2_{A} + {\|T + T^{\sharp_{A}}\|}^2_{A}\Big).
\end{align*}
Taking the supremum over $x \in \mathcal{H}$ with ${\|x\|}_{A} = 1$ in the above inequality, we get
\begin{align*}
w^2_{A}(T) \leq
\frac{1}{4} \Big({\|T - T^{\sharp_{A}}\|}^2_{A} + {\|T + T^{\sharp_{A}}\|}^2_{A}\Big),
\end{align*}
and hence
\begin{align}\label{I.3.T.7}
w_{A}(T) \leq
\frac{1}{2} \Big({\|T - T^{\sharp_{A}}\|}^2_{A} + {\|T + T^{\sharp_{A}}\|}^2_{A}\Big)^{1/2}.
\end{align}
From (\ref{I.2.T.7}) and (\ref{I.3.T.7}), we deduce the desired result.
\end{proof}
\begin{remark}\label{R.7.1}
The first inequality in Proposition \ref{P.7} has recently been proved by a different way in \cite[Corollary 2.7]{Z.3}.
\end{remark}
\begin{theorem}\label{T.8}
Let $T\in\mathbb{B}_{A}(\mathcal{H})$. Then
\begingroup\makeatletter\def\f@size{10}\check@mathfonts
\begin{align*}
\frac{1}{2} \max\Big\{{\|T^2 - (T^{\sharp_{A}})^2\|}^{1/2}_{A}, {\|T^2 + (T^{\sharp_{A}})^2\|}^{1/2}_{A}\Big\}
\leq w_{A}(T) \leq
\frac{\sqrt{2}}{2} \Big({\|T\|}^2_{A} + w_{A}(T^2)\Big)^{1/2}.
\end{align*}
\endgroup
\end{theorem}
\begin{proof}
It follows from (\ref{I.2.9.T.7}) that $w^2_{A}(T) \geq \frac{1}{4}{\|T \pm T^{\sharp_{A}}\|}^2_{A}$.
Therefore,
\begin{align*}
w^2_{A}(T) &\geq
\frac{1}{8}\Big({\|T + T^{\sharp_{A}}\|}^2_{A}
+ {\|T - T^{\sharp_{A}}\|}^2_{A}\Big)
\\& \geq \frac{1}{8}\Big({\|(T + T^{\sharp_{A}})^2\|}_{A}
+ {\|(T - T^{\sharp_{A}})^2\|}_{A}\Big)
\\& \geq \frac{1}{8}{\|(T + T^{\sharp_{A}})^2 + (T - T^{\sharp_{A}})^2\|}_{A}\\
&= \frac{1}{4}{\|T^2 + (T^{\sharp_{A}})^2\|}_{A},
\end{align*}
and so
\begin{align}\label{I.1.T.8}
\frac{1}{2} {\|T^2 + (T^{\sharp_{A}})^2\|}^{1/2}_{A} \leq w_{A}(T).
\end{align}
Utilizing a similar argument as in Proposition \ref{P.7}, we get
$w^2_{A}(T) \geq \frac{1}{4}{\|T \pm iT^{\sharp_{A}}\|}^2_{A}$.
Hence,
\begin{align*}
w^2_{A}(T) &\geq
\frac{1}{8}\Big({\|T + iT^{\sharp_{A}}\|}^2_{A}
+ {\|T - iT^{\sharp_{A}}\|}^2_{A}\Big)
\\& \geq \frac{1}{8}\Big({\|(T + iT^{\sharp_{A}})^2\|}_{A}
+ {\|(T - iT^{\sharp_{A}})^2\|}_{A}\Big)
\\& \geq \frac{1}{8}{\big\|(T + iT^{\sharp_{A}})^2 + (T - iT^{\sharp_{A}})^2\big\|}_{A}\\
&= \frac{1}{4}{\|T^2 - (T^{\sharp_{A}})^2\|}_{A}.
\end{align*}
Hence
\begin{align}\label{I.2.T.8}
\frac{1}{2} {\|T^2 - (T^{\sharp_{A}})^2\|}^{1/2}_{A} \leq w_{A}(T).
\end{align}
From (\ref{I.1.T.8}) and (\ref{I.2.T.8}), we conclude that
\begin{align*}
\frac{1}{2} \max\Big\{{\|T^2 - (T^{\sharp_{A}})^2\|}^{1/2}_{A}, {\|T^2 + (T^{\sharp_{A}})^2\|}^{1/2}_{A}\Big\}
\leq w_{A}(T).
\end{align*}
Now, let $x \in \mathcal{H}$ with ${\|x\|}_{A} = 1$. Then
\begin{align*}
|{\langle Tx, x\rangle}_{A}|^2 &= \frac{1}{2}\Big(|{\langle Tx, x\rangle}_{A}|^2
+ |{\langle x, T^{\sharp_{A}}x\rangle}_{A}|^2\Big)
\\& = \frac{1}{2}{\Big\langle x, {\langle x, Tx\rangle}_{A} Tx
+ {\langle x, T^{\sharp_{A}}x\rangle}_{A} T^{\sharp_{A}}x\Big\rangle}_{A}
\\& \leq \frac{1}{2}{\Big\|{\langle x, Tx\rangle}_{A} Tx
+ {\langle x, T^{\sharp_{A}}x\rangle}_{A} T^{\sharp_{A}}x\Big\|}_{A}
\\& = \frac{1}{2}\Big(|{\langle x, Tx\rangle}_{A}|^2 {\|Tx\|}^{2}_{A}
+ |{\langle x, T^{\sharp_{A}}x\rangle}_{A}|^2 {\|T^{\sharp_{A}}x\|}^2_{A}
\\& \qquad \qquad
+ 2\mbox{Re}\Big({\langle x, Tx\rangle}_{A}{\langle T^{\sharp_{A}}x, x\rangle}_{A}{\langle Tx, T^{\sharp_{A}}x\rangle}_{A}\Big)\Big)^{1/2}
\\& \leq \frac{1}{2}\Big(|{\langle x, Tx\rangle}_{A}|^2 {\|T\|}^{2}_{A}
+ |{\langle x, T^{\sharp_{A}}x\rangle}_{A}|^2 {\|T^{\sharp_{A}}\|}^2_{A}
\\& \qquad \qquad
+ 2|{\langle x, Tx\rangle}_{A}|\,|{\langle T^{\sharp_{A}}x, x\rangle}_{A}|\,|{\langle Tx, T^{\sharp_{A}}x\rangle}_{A}|\Big)^{1/2}
\\& = \frac{1}{2}\Big(|{\langle Tx, x\rangle}_{A}|^2 {\|T\|}^{2}_{A}
+ |{\langle Tx, x\rangle}_{A}|^2 {\|T\|}^2_{A}
+ 2|{\langle Tx, x\rangle}_{A}|^2\,|{\langle T^2x, x\rangle}_{A}|\Big)^{1/2}
\\& = \frac{\sqrt{2}}{2}|{\langle Tx, x\rangle}_{A}| \Big({\|T\|}^2_{A} + |{\langle T^2x, x\rangle}_{A}|\Big)^{1/2}
\\& \leq \frac{\sqrt{2}}{2}|{\langle Tx, x\rangle}_{A}| \Big({\|T\|}^2_{A} + w_{A}(T^2)\Big)^{1/2}.
\end{align*}
Thus
\begin{align}\label{I.3.9.T.8}
|{\langle Tx, x\rangle}_{A}| \leq \frac{\sqrt{2}}{2}\Big({\|T\|}^2_{A} + w_{A}(T^2)\Big)^{1/2}.
\end{align}
Clearly this inequality holds also when ${\langle Tx, x\rangle}_{A} = 0$.
Taking the supremum over unit vectors $x \in \mathcal{H}$ in (\ref{I.3.9.T.8}), we deduce that
\begin{align*}
w_{A}(T) \leq \frac{\sqrt{2}}{2} \Big({\|T\|}^2_{A} + w_{A}(T^2)\Big)^{1/2}.
\end{align*}
\end{proof}
\begin{remark}\label{R.8.1}
From the above theorem and (\ref{I.S1.1}), we obviously have
\begin{align*}
w_{A}(T) \leq \frac{\sqrt{2}}{2} \Big({\|T\|}^2_{A} + w_{A}(T^2)\Big)^{1/2} \leq {\|T\|}_{A}.
\end{align*}
Thus the second inequality obtained by us in Theorem \ref{T.8} improves the second inequality of (\ref{I.S1.1}).
The following example shows that it is a nontrivial improvement.
Consider $A = \begin{bmatrix}
1 & 0 \\
0 & 2
\end{bmatrix}$ and
$T = \begin{bmatrix}
1 & 2 \\
0 & 1
\end{bmatrix}$.
Then simple computations show that ${\|T\|}_{A} = \sqrt{2 + \sqrt{3}}, w_{A}(T) = \frac{2 + \sqrt{2}}{2}$, and $w_{A}(T^2) = 1 + \sqrt{2}$.
Thus
\begin{align*}
w_{A}(T) \simeq 1.71 <
\frac{\sqrt{2}}{2} \Big({\|T\|}^2_{A} + w_{A}(T^2)\Big)^{1/2} \simeq 1.75
< {\|T\|}_{A} \simeq 1.93.
\end{align*}
\end{remark}
The following lemma will be useful in the proof of the next result.
\begin{lemma}\label{L.9.0}
Let $X, Y\in\mathbb{B}(\mathcal{H})$. Then
\begin{align*}
(XY + YX)^2 + (X^2 + Y^2)^2 = \frac{1}{2}(X + Y)^4 + \frac{1}{2}(X - Y)^4.
\end{align*}
\end{lemma}
\begin{proof}
The proof is trivial.
\end{proof}
\begin{theorem}\label{T.9}
Let $T\in\mathbb{B}_{A}(\mathcal{H})$. Then
\begingroup\makeatletter\def\f@size{9}\check@mathfonts
\begin{align*}
\Big(\frac{1}{16} {\|TT^{\sharp_{A}} + T^{\sharp_{A}}T\|}^{2}_{A}
+ \frac{1}{16}c_{A}\Big(\big(T^2 + (T^{\sharp_{A}})^2\big)^2\Big)\,\Big)^{1/4}
\leq w_{A}(T) \leq
\Big(\frac{1}{8} {\|TT^{\sharp_{A}} + T^{\sharp_{A}}T\|}^{2}_{A} + \frac{1}{2}w^2_{A}(T^2)\Big)^{1/4}.
\end{align*}
\endgroup
\end{theorem}
\begin{proof}
Let $x \in \mathcal{H}$ with ${\|x\|}_{A} = 1$.
Applying Lemma \ref{L.9.0}  with $X:=T^{\sharp_{A}}$ and $Y:=(T^{\sharp_{A}})^{\sharp_{A}}$, we get
\begingroup\makeatletter\def\f@size{10}\check@mathfonts
\begin{align*}
\big(T^{\sharp_{A}}(T^{\sharp_{A}})^{\sharp_{A}} + (T^{\sharp_{A}})^{\sharp_{A}}T^{\sharp_{A}}\big)^2
&+ \big((T^{\sharp_{A}})^2 + ((T^{\sharp_{A}})^{\sharp_{A}})^2\big)^2
\\& = \frac{1}{2}\big(T^{\sharp_{A}} + (T^{\sharp_{A}})^{\sharp_{A}}\big)^4
+ \frac{1}{2}\big(T^{\sharp_{A}} - (T^{\sharp_{A}})^{\sharp_{A}}\big)^4.
\end{align*}
\endgroup
Thus
\begingroup\makeatletter\def\f@size{9}\check@mathfonts
\begin{align*}
{\big\langle \big(T^{\sharp_{A}}(T^{\sharp_{A}})^{\sharp_{A}}
+ (T^{\sharp_{A}})^{\sharp_{A}}T^{\sharp_{A}}\big)^2x, x\big\rangle}_{A}
&+ {\big\langle \big((T^{\sharp_{A}})^2 + ((T^{\sharp_{A}})^{\sharp_{A}})^2\big)^2x, x\big\rangle}_{A}
\\& = \frac{1}{2}{\big\langle \big(T^{\sharp_{A}} + (T^{\sharp_{A}})^{\sharp_{A}}\big)^4x, x\big\rangle}_{A}
+ \frac{1}{2}{\big\langle \big(T^{\sharp_{A}} - (T^{\sharp_{A}})^{\sharp_{A}}\big)^4x, x\big\rangle}_{A}.
\end{align*}
\endgroup
This implies that
\begingroup\makeatletter\def\f@size{10}\check@mathfonts
\begin{align*}
{\big\langle \big(T^{\sharp_{A}}(T^{\sharp_{A}})^{\sharp_{A}}
+ (T^{\sharp_{A}})^{\sharp_{A}}T^{\sharp_{A}}\big)^2x, x\big\rangle}_{A}
&+ c_{A}\Big(\big((T^{\sharp_{A}})^2 + ((T^{\sharp_{A}})^{\sharp_{A}})^2\big)^2\Big)
\\& \leq
\frac{1}{2}{\Big\|\big(T^{\sharp_{A}} + (T^{\sharp_{A}})^{\sharp_{A}}\big)^4\Big\|}_{A}
+ \frac{1}{2}{\Big\|\big(T^{\sharp_{A}} - (T^{\sharp_{A}})^{\sharp_{A}}\big)^4\Big\|}_{A}.
\end{align*}
\endgroup
Therefore, we get
\begingroup\makeatletter\def\f@size{10}\check@mathfonts
\begin{align*}
{\big\langle \big(T^{\sharp_{A}}(T^{\sharp_{A}})^{\sharp_{A}}
+ (T^{\sharp_{A}})^{\sharp_{A}}T^{\sharp_{A}}\big)^2x, x\big\rangle}_{A}
&+ c_{A}\Big(\big(T^2 + (T^{\sharp_{A}})^2\big)^2\Big)
\\& \leq \frac{1}{2}{\|T^{\sharp_{A}} + (T^{\sharp_{A}})^{\sharp_{A}}\|}^4_{A}
+ \frac{1}{2}{\|T^{\sharp_{A}} - (T^{\sharp_{A}})^{\sharp_{A}}\|}^4_{A}
\\& = 8{\left\|\frac{T + T^{\sharp_{A}}}{2}\right\|}^4_{A}
+ 8{\left\|\frac{T - T^{\sharp_{A}}}{2}\right\|}^4_{A}.
\end{align*}
\endgroup
It follows from Proposition \ref{P.7} that
\begingroup\makeatletter\def\f@size{10}\check@mathfonts
\begin{align*}
{\big\langle \big(T^{\sharp_{A}}(T^{\sharp_{A}})^{\sharp_{A}}
+ (T^{\sharp_{A}})^{\sharp_{A}}T^{\sharp_{A}}\big)^2x, x\big\rangle}_{A}
+ c_{A}\Big(\big(T^2 + (T^{\sharp_{A}})^2\big)^2\Big)
\leq 8 w^4_{A}(T) + 8 w^4_{A}(T) = 16w^4_{A}(T).
\end{align*}
\endgroup
Taking the supremum over $x \in \mathcal{H}$ with ${\|x\|}_{A} = 1$ in the above inequality, we obtain
\begin{align}\label{I.1.T.9}
\Big(\frac{1}{16} {\|TT^{\sharp_{A}} + T^{\sharp_{A}}T\|}^{2}_{A}
+ \frac{1}{16}c_{A}\Big(\big(T^2 + (T^{\sharp_{A}})^2\big)^2\Big)\,\Big)^{1/4}
\leq w_{A}(T).
\end{align}
Now, let $\theta \in \mathbb{R}$.
Then by letting $X = (e^{i\theta}T)^{\sharp_{A}}$ and $Y = ((e^{i\theta}T)^{\sharp_{A}})^{\sharp_{A}}$
in Lemma \ref{L.9.0}, we have
\begingroup\makeatletter\def\f@size{9}\check@mathfonts
\begin{align*}
\Big((e^{i\theta}T)^{\sharp_{A}}((e^{i\theta}T)^{\sharp_{A}})^{\sharp_{A}}
+ ((e^{i\theta}T)^{\sharp_{A}})^{\sharp_{A}}&(e^{i\theta}T)^{\sharp_{A}}\Big)^2
+ \Big(((e^{i\theta}T)^{\sharp_{A}})^2 + \big(((e^{i\theta}T)^{\sharp_{A}})^{\sharp_{A}}\big)^2\Big)^2
\\& = \frac{1}{2}\Big((e^{i\theta}T)^{\sharp_{A}} + ((e^{i\theta}T)^{\sharp_{A}})^{\sharp_{A}}\Big)^4
+ \frac{1}{2}\Big((e^{i\theta}T)^{\sharp_{A}} - ((e^{i\theta}T)^{\sharp_{A}})^{\sharp_{A}}\Big)^4.
\end{align*}
\endgroup
Since
\begin{align*}
(e^{i\theta}T)^{\sharp_{A}}((e^{i\theta}T)^{\sharp_{A}})^{\sharp_{A}}
+ ((e^{i\theta}T)^{\sharp_{A}})^{\sharp_{A}}(e^{i\theta}T)^{\sharp_{A}}
= T^{\sharp_{A}}(T^{\sharp_{A}})^{\sharp_{A}}
+ (T^{\sharp_{A}})^{\sharp_{A}}T^{\sharp_{A}}
\end{align*}
and $\frac{1}{2}\Big((e^{i\theta}T)^{\sharp_{A}} - ((e^{i\theta}T)^{\sharp_{A}})^{\sharp_{A}}\Big)^4$
is an $A$-positive operator, we have
\begingroup\makeatletter\def\f@size{9}\check@mathfonts
\begin{align*}
\frac{1}{2}\Big((e^{i\theta}T)^{\sharp_{A}} + ((e^{i\theta}T)^{\sharp_{A}})^{\sharp_{A}}\Big)^4
{\leq}_{A} \Big(T^{\sharp_{A}}(T^{\sharp_{A}})^{\sharp_{A}}
+ (T^{\sharp_{A}})^{\sharp_{A}}T^{\sharp_{A}}\Big)^2
+ \Big(((e^{i\theta}T)^{\sharp_{A}})^2 + \big(((e^{i\theta}T)^{\sharp_{A}})^{\sharp_{A}}\big)^2\Big)^2.
\end{align*}
\endgroup
Thus
\begingroup\makeatletter\def\f@size{7}\check@mathfonts
\begin{align*}
\frac{1}{2}{\Big\|\Big((e^{i\theta}T)^{\sharp_{A}} + ((e^{i\theta}T)^{\sharp_{A}})^{\sharp_{A}}\Big)^4\Big\|}_{A}
\leq {\Big\|\Big(T^{\sharp_{A}}(T^{\sharp_{A}})^{\sharp_{A}}
+ (T^{\sharp_{A}})^{\sharp_{A}}T^{\sharp_{A}}\Big)^2\Big\|}_{A}
+ {\Big\|\Big(((e^{i\theta}T)^{\sharp_{A}})^2 + \big(((e^{i\theta}T)^{\sharp_{A}})^{\sharp_{A}}\big)^2\Big)^2\Big\|}_{A}.
\end{align*}
\endgroup
Hence
\begingroup\makeatletter\def\f@size{9}\check@mathfonts
\begin{align*}
\frac{1}{2}{\Big\|(e^{i\theta}T)^{\sharp_{A}} + ((e^{i\theta}T)^{\sharp_{A}})^{\sharp_{A}}\Big\|}^4_{A}
\leq {\Big\|T^{\sharp_{A}}(T^{\sharp_{A}})^{\sharp_{A}}
+ (T^{\sharp_{A}})^{\sharp_{A}}T^{\sharp_{A}}\Big\|}^2_{A}
+ {\Big\|((e^{i\theta}T)^{\sharp_{A}})^2
+ \big(((e^{i\theta}T)^{\sharp_{A}})^{\sharp_{A}}\big)^2\Big\|}^2_{A},
\end{align*}
\endgroup
or equivalently by (\ref{I.S1.0}),
\begin{align*}
{\Big\|\frac{e^{i\theta}T + (e^{i\theta}T)^{\sharp_{A}}}{2}\Big\|}^4_{A}
\leq \frac{1}{8}{\Big\|TT^{\sharp_{A}} + T^{\sharp_{A}}T\Big\|}^2_{A}
+ \frac{1}{2}{\Big\|\frac{(e^{i\theta}T)^2 + \big((e^{i\theta}T)^{\sharp_{A}}\big)^2}{2}\Big\|}^2_{A}.
\end{align*}
This together with (\ref{I.S1.2}) gives
\begin{align*}
w^4_{A}(T) \leq \frac{1}{8} {\|TT^{\sharp_{A}} + T^{\sharp_{A}}T\|}^{2}_{A} + \frac{1}{2}w^2_{A}(T^2),
\end{align*}
whence
\begin{align}\label{I.2.T.9}
w_{A}(T) \leq
\Big(\frac{1}{8} {\|TT^{\sharp_{A}} + T^{\sharp_{A}}T\|}^{2}_{A} + \frac{1}{2}w^2_{A}(T^2)\Big)^{1/4}.
\end{align}
By (\ref{I.1.T.9}) and (\ref{I.2.T.9}), we deduce the desired result.
\end{proof}
\begin{remark}\label{R.9.1}
Since ${\|T\|}^2_{A} \leq {\|TT^{\sharp_{A}} + T^{\sharp_{A}}T\|}_{A} \leq 2{\|T\|}^2_{A}$
and $w_{A}(T^2)\leq {\|T\|}^2_{A}$, we have
\begin{align*}
\frac{1}{2}{\|T\|}_{A} \leq \Big(\frac{1}{16} {\|TT^{\sharp_{A}} + T^{\sharp_{A}}T\|}^{2}_{A}
+ \frac{1}{16}c_{A}\Big(\big(T^2 + (T^{\sharp_{A}})^2\big)^2\Big)\,\Big)^{1/4}
\end{align*}
and
\begin{align*}
\Big(\frac{1}{8} {\|TT^{\sharp_{A}} + T^{\sharp_{A}}T\|}^{2}_{A} + \frac{1}{2}w^2_{A}(T^2)\Big)^{1/4} \leq {\|T\|}_{A}.
\end{align*}
So, the inequalities in Theorem \ref{T.9} improve inequalities (\ref{I.S1.1}).
To see this, let $A = \begin{bmatrix}
1 & -1 \\
-1 & 2
\end{bmatrix}$ and
$T = \begin{bmatrix}
1 & 0 \\
1 & 1
\end{bmatrix}$.
Easy computations show that $w_{A}(T) = 2$, $w_{A}(T^2) = 3$,
$c_{A}\Big(\big(T^2 + (T^{\sharp_{A}})^2\big)^2\Big) = 4$,
${\|T\|}_{A} = \sqrt{3 + 2\sqrt{2}}$, and ${\|TT^{\sharp_{A}} + T^{\sharp_{A}}T\|}_{A} = 10$.
Hence
\begingroup\makeatletter\def\f@size{10}\check@mathfonts
\begin{align*}
\frac{1}{2}{\|T\|}_{A} \simeq 1.21 < \Big(\frac{1}{16} {\|TT^{\sharp_{A}} + T^{\sharp_{A}}T\|}^{2}_{A}
+ \frac{1}{16}c_{A}\Big(\big(T^2 + (T^{\sharp_{A}})^2\big)^2\Big)\,\Big)^{1/4} \simeq 1.60 < w_{A}(T) = 2
\end{align*}
\endgroup
and
\begingroup\makeatletter\def\f@size{10}\check@mathfonts
\begin{align*}
w_{A}(T) = 2 <
\Big(\frac{1}{8} {\|TT^{\sharp_{A}} + T^{\sharp_{A}}T\|}^{2}_{A} + \frac{1}{2}w^2_{A}(T^2)\Big)^{1/4} \simeq 2.03
< {\|T\|}_{A} \simeq 2.41.
\end{align*}
\endgroup
\end{remark}
\begin{remark}\label{R.9.2}
Very recently, as our work was in progress, the second inequality
in Theorem \ref{T.9} has been proved by Bhunia, Paul, and  Nayak in \cite{B.P.N}.
Our approach here is different from theirs.
\end{remark}
Let $T\in\mathbb{B}_{A}(\mathcal{H})$.
By inequalities (\ref{I.S1.1}), we have
\begin{align}\label{I.0.T.10}
w_{A}(T^2) \leq {\|T^2\|}_{A} \leq {\|T\|}^2_{A} \leq 4w^2_{A}(T).
\end{align}
In the following result, we improve inequalities (\ref{I.0.T.10}).
\begin{proposition}\label{P.11}
Let $T\in\mathbb{B}_{A}(\mathcal{H})$. Then
$w_{A}(T^2) \leq w^2_{A}(T)$.
\end{proposition}
\begin{proof}
First, let us show that
\begin{align}\label{I.0.T.11}
{\|Tx\|}^2_{A} + |{\langle T^2x, x\rangle}_{A}| \leq 2 w_{A}(T){\|Tx\|}_{A} \qquad (x \in \mathcal{H}, {\|x\|}_{A} = 1).
\end{align}
Let $x \in \mathcal{H}$ with ${\|x\|}_{A} = 1$.
We consider two cases.

Case 1: ${\|Tx\|}_{A} = 0$. Then
\begin{align*}
|{\langle T^2x, x\rangle}_{A}| = |{\langle Tx, T^{\sharp_{A}}x\rangle}_{A}| \leq {\|Tx\|}_{A}\, {\|T^{\sharp_{A}}x\|}_{A} = 0,
\end{align*}
and so $|{\langle T^2x, x\rangle}_{A}| = 0$. Thus (\ref{I.0.T.11}) is satisfied.

Case 2: ${\|Tx\|}_{A} \neq 0$.
Let $\theta \in \mathbb{R}$ such that ${\langle T^2x, x\rangle}_{A} = e^{i\theta} |{\langle T^2x, x\rangle}_{A}|$.
We have
\begingroup\makeatletter\def\f@size{10}\check@mathfonts
\begin{align*}
{\|Tx\|}_{A} + \frac{|{\langle T^2x, x\rangle}_{A}|}{{\|Tx\|}_{A}}
&= \frac{{\langle Tx, Tx\rangle}_{A}}{{\|Tx\|}_{A}} + \frac{e^{-i\theta} {\langle T^2x, x\rangle}_{A}}{{\|Tx\|}_{A}}
\\& = \left|e^{i\frac{\theta}{2}}{\left\langle Tx, T\left(\frac{x}{{\|Tx\|}_{A}}\right)\right\rangle}_{A}
+ e^{i\frac{\theta}{2}}{\left\langle T^2\left(\frac{e^{-i\theta}x}{{\|Tx\|}_{A}}\right), x\right\rangle}_{A}\right|
\\& = \frac{1}{2}\Big|{\left\langle T\left(x + e^{-i\frac{\theta}{2}}\frac{Tx}{{\|Tx\|}_{A}}\right),
\left(x + e^{-i\frac{\theta}{2}}\frac{Tx}{{\|Tx\|}_{A}}\right)\right\rangle}_{A}
\\& \qquad \qquad \qquad - {\left\langle T\left(x - e^{-i\frac{\theta}{2}}\frac{Tx}{{\|Tx\|}_{A}}\right),
\left(x - e^{-i\frac{\theta}{2}}\frac{Tx}{{\|Tx\|}_{A}}\right)\right\rangle}_{A}\Big|
\\& \leq \frac{1}{2}\left|{\left\langle T\left(x + e^{-i\frac{\theta}{2}}\frac{Tx}{{\|Tx\|}_{A}}\right),
\left(x + e^{-i\frac{\theta}{2}}\frac{Tx}{{\|Tx\|}_{A}}\right)\right\rangle}_{A}\right|
\\& \qquad \qquad + \frac{1}{2}\left|{\left\langle T\left(x - e^{-i\frac{\theta}{2}}\frac{Tx}{{\|Tx\|}_{A}}\right),
\left(x - e^{-i\frac{\theta}{2}}\frac{Tx}{{\|Tx\|}_{A}}\right)\right\rangle}_{A}\right|
\\& \leq \frac{1}{2}w_{A}(T){\left\|x + e^{-i\frac{\theta}{2}}\frac{Tx}{{\|Tx\|}_{A}}\right\|}^2_{A}
+ \frac{1}{2}w_{A}(T){\left\|x - e^{-i\frac{\theta}{2}}\frac{Tx}{{\|Tx\|}_{A}}\right\|}^2_{A}
\\& = \frac{1}{2}w_{A}(T)\Big({\left\|x + e^{-i\frac{\theta}{2}}\frac{Tx}{{\|Tx\|}_{A}}\right\|}^2_{A}
+ {\left\|x - e^{-i\frac{\theta}{2}}\frac{Tx}{{\|Tx\|}_{A}}\right\|}^2_{A}\Big)
\\& = 2w_{A}(T),
\end{align*}
\endgroup
and hence
${\|Tx\|}_{A} + \frac{|{\langle T^2x, x\rangle}_{A}|}{{\|Tx\|}_{A}} \leq 2w_{A}(T)$.
Therefore (\ref{I.0.T.11}) is satisfied.

Now, from (\ref{I.0.T.11}) for $x \in \mathcal{H}$ with ${\|x\|}_{A} = 1$, we get
\begin{align*}
0 &\leq 2 w_{A}(T){\|Tx\|}_{A} - {\|Tx\|}^2_{A} - |{\langle T^2x, x\rangle}_{A}|
\\& = w^2_{A}(T) - \big( w_{A}(T) - {\|Tx\|}_{A}\big)^2 - |{\langle T^2x, x\rangle}_{A}|
\\& \leq w^2_{A}(T) - |{\langle T^2x, x\rangle}_{A}|,
\end{align*}
which implies $|{\langle T^2x, x\rangle}_{A}| \leq  w^2_{A}(T)$.
Taking the supremum over unit vectors $x \in \mathcal{H}$, we deduce that
$w_{A}(T^2) \leq w^2_{A}(T)$.
\end{proof}
The following result is a reverse type inequality of the inequality in Proposition \ref{P.11}.
\begin{theorem}\label{T.11.5}
Let $T\in\mathbb{B}_{A}(\mathcal{H})$. Then
\begin{align*}
w^2_{A}(T) \leq w_{A}(T^2) + \frac{1}{2} \min\Big\{{\|T - T^{\sharp_{A}}\|}^2_{A}, {\|T + T^{\sharp_{A}}\|}^2_{A}\Big\}.
\end{align*}
\end{theorem}
\begin{proof}
First observe that, by \cite[Theorem 2.10]{Z.3}, we have
\begin{align}\label{I.1.T.11.5}
2w^2_{A}(T) \leq {\|TT^{\sharp_{A}} + T^{\sharp_{A}}T\|}_{A}.
\end{align}
Moreover,
\begingroup\makeatletter\def\f@size{9}\check@mathfonts
\begin{align*}
{\|TT^{\sharp_{A}} + T^{\sharp_{A}}T\|}_{A}
& = {\|(T^{\sharp_{A}})^{\sharp_{A}}T^{\sharp_{A}} + T^{\sharp_{A}} (T^{\sharp_{A}})^{\sharp_{A}}\|}_{A}
\qquad \qquad \Big(\mbox{by (\ref{I.S1.0})}\Big)
\\& = w_{A}\Big((T^{\sharp_{A}})^{\sharp_{A}}T^{\sharp_{A}} + T^{\sharp_{A}} (T^{\sharp_{A}})^{\sharp_{A}}\Big)
\\& \qquad \qquad \Big(\mbox{since $(T^{\sharp_{A}})^{\sharp_{A}}T^{\sharp_{A}} + T^{\sharp_{A}} (T^{\sharp_{A}})^{\sharp_{A}}$
is an $A$-selfadjoint operator}\Big)
\\& = w_{A}\Big(\big(T^{\sharp_{A}} \pm (T^{\sharp_{A}})^{\sharp_{A}}\big)^{\sharp_{A}}\big(T^{\sharp_{A}} \pm (T^{\sharp_{A}})^{\sharp_{A}}\big)
\mp \big((T^{\sharp_{A}})^{\sharp_{A}}(T^{\sharp_{A}})^{\sharp_{A}} + T^{\sharp_{A}}T^{\sharp_{A}}\big)\Big)
\\& \leq w_{A}\Big(\big(T^{\sharp_{A}} \pm (T^{\sharp_{A}})^{\sharp_{A}}\big)^{\sharp_{A}}\big(T^{\sharp_{A}} \pm (T^{\sharp_{A}})^{\sharp_{A}}\big)\Big)
+ w_{A}\Big((T^{\sharp_{A}})^{\sharp_{A}}(T^{\sharp_{A}})^{\sharp_{A}}\Big)
+ w_{A}(T^{\sharp_{A}}T^{\sharp_{A}}\big)
\\& = {\Big\|\big(T^{\sharp_{A}} \pm (T^{\sharp_{A}})^{\sharp_{A}}\big)^{\sharp_{A}}\big(T^{\sharp_{A}} \pm (T^{\sharp_{A}})^{\sharp_{A}}\big)\Big\|}_{A}
+ w_{A}(T^2) + w_{A}(T^2)
\\&\qquad \big(\mbox{since
$\big(T^{\sharp_{A}} \pm (T^{\sharp_{A}})^{\sharp_{A}}\big)^{\sharp_{A}}\big(T^{\sharp_{A}} \pm (T^{\sharp_{A}})^{\sharp_{A}}\big)$
is an $A$-selfadjoint operator}\big)
\\& = {\Big\|T^{\sharp_{A}} \pm (T^{\sharp_{A}})^{\sharp_{A}}\Big\|}^2_{A}
+ 2 w_{A}(T^2)
= {\Big\|T \pm T^{\sharp_{A}}\Big\|}^2_{A} + 2 w_{A}(T^2), \qquad \Big(\mbox{by (\ref{I.S1.0})}\Big)
\end{align*}
\endgroup
and so
\begin{align}\label{I.2.T.11.5}
{\|TT^{\sharp_{A}} + T^{\sharp_{A}}T\|}_{A} \leq 2 w_{A}(T^2) + {\Big\|T \pm T^{\sharp_{A}}\Big\|}^2_{A}.
\end{align}
Finally, by (\ref{I.1.T.11.5}) and (\ref{I.2.T.11.5}), we conclude that
\begin{align*}
w^2_{A}(T) \leq w_{A}(T^2) + \frac{1}{2} \min\Big\{{\|T - T^{\sharp_{A}}\|}^2_{A}, {\|T + T^{\sharp_{A}}\|}^2_{A}\Big\}.
\end{align*}
\end{proof}
For $T, S\in\mathbb{B}_{A}(\mathcal{H})$, we clearly have $w_{A}(T + S) \leq w_{A}(T) + w_{A}(S)$.
The following theorem deals with the equality $w_{A}(T + S) = w_{A}(T) + w_{A}(S)$.
\begin{theorem}\label{T.12}
Let $T, S\in\mathbb{B}_{A}(\mathcal{H})$. Then the following conditions are equivalent:
\begin{itemize}
\item[(i)] $w_{A}(T + S) = w_{A}(T) + w_{A}(S)$.
\item[(ii)] There exists a sequence of $A$-unit vectors $\{x_n\}$ in $\mathcal{H}$ such that
\begin{align*}
\displaystyle{\lim_{n\rightarrow +\infty}}{\langle x_n, Tx_n\rangle}_{A} {\langle Sx_n, x_n\rangle}_{A}
= w_{A}(T)\,w_{A}(S).
\end{align*}
\end{itemize}
\end{theorem}
\begin{proof}
The proof is similar to that of Theorem \ref{T.4} and so we omit it.
\end{proof}
If $T, S\in\mathbb{B}_{A}(\mathcal{H})$, then Proposition \ref{P.11} ensures that
\begin{align*}
w_{A}(T^2 + S^2) \leq w_{A}(T^2) + w_{A}(S^2)
\leq w^2_{A}(T) + w^2_{A}(S)
\leq 2\max\big\{w^2_{A}(T), w^2_{A}(S)\big\},
\end{align*}
and hence
\begin{align}\label{I.0.T.13}
w_{A}(T^2 + S^2) \leq 2\max\big\{w^2_{A}(T), w^2_{A}(S)\big\}.
\end{align}
Finally, we state a condition for the equality in (\ref{I.0.T.13}) by applying Theorem \ref{T.12}.
\begin{theorem}\label{T.13}
Let $T, S\in\mathbb{B}_{A}(\mathcal{H})$. Then the following conditions are equivalent:
\begin{itemize}
\item[(i)] $w_{A}(T^2 + S^2) = 2\max\big\{w^2_{A}(T), w^2_{A}(S)\big\}$.
\item[(ii)] There exists a sequence of $A$-unit vectors $\{x_n\}$ in $\mathcal{H}$ such that
\begin{align*}
\displaystyle{\lim_{n\rightarrow +\infty}}{\langle x_n, T^2x_n\rangle}_{A} {\langle S^2x_n, x_n\rangle}_{A}
= \max\big\{w^4_{A}(T), w^4_{A}(S)\big\}.
\end{align*}
\end{itemize}
\end{theorem}
\begin{proof}
Let $w_{A}(T^2 + S^2) = 2\max\big\{w^2_{A}(T), w^2_{A}(S)\big\}$.
From the derivation of (\ref{I.0.T.13}), we have
\begin{align*}
w_{A}(T^2) + w_{A}(S^2) = w^2_{A}(T) + w^2_{A}(S) = 2\max\big\{w^2_{A}(T), w^2_{A}(S)\big\}.
\end{align*}
Hence $w_{A}(T^2) = w_{A}(S^2) = w^2_{A}(T) = w^2_{A}(S)$. Thus
\begin{align*}
w_{A}(T^2 + S^2) = w_{A}(T^2) + w_{A}(S^2).
\end{align*}
By Theorem \ref{T.12}, there exists a sequence of $A$-unit vectors $\{x_n\}$ in $\mathcal{H}$ such that
\begin{align*}
\displaystyle{\lim_{n\rightarrow +\infty}}{\langle x_n, T^2x_n\rangle}_{A} {\langle S^2x_n, x_n\rangle}_{A}
= w_{A}(T^2)\,w_{A}(S^2).
\end{align*}
Since $w_{A}(T^2) = w_{A}(S^2) = w^2_{A}(T) = w^2_{A}(S)$, we have
\begin{align*}
\displaystyle{\lim_{n\rightarrow +\infty}}{\langle x_n, T^2x_n\rangle}_{A} {\langle S^2x_n, x_n\rangle}_{A}
= \max\big\{w^4_{A}(T), w^4_{A}(S)\big\}.
\end{align*}

Conversely, assume that there exists a sequence of $A$-unit vectors $\{x_n\}$ in $\mathcal{H}$ such that
\begin{align*}
\displaystyle{\lim_{n\rightarrow +\infty}}{\langle x_n, T^2x_n\rangle}_{A} {\langle S^2x_n, x_n\rangle}_{A}
= \max\big\{w^4_{A}(T), w^4_{A}(S)\big\}.
\end{align*}
Then
\begin{align*}
\max\big\{w^4_{A}(T), w^4_{A}(S)\big\}\leq w_{A}(T^2) w_{A}(S^2).
\end{align*}
Hence, by Proposition \ref{P.11} and the arithmetic-geometric mean inequality, we have
\begin{align*}
\max\big\{w^4_{A}(T), w^4_{A}(S)\big\}&\leq w_{A}(T^2) w_{A}(S^2)
\\&\leq w^2_{A}(T) w^2_{A}(S)
\\&\leq \frac{w^4_{A}(T) + w^4_{A}(S)}{2}
\leq \max\big\{w^4_{A}(T), w^4_{A}(S)\big\}.
\end{align*}
Thus $w_{A}(T^2) = w^2_{A}(T) = w^2_{A}(S) = w_{A}(S^2)$ and consequently,
\begin{align*}
\displaystyle{\lim_{n\rightarrow +\infty}}{\langle x_n, T^2x_n\rangle}_{A} {\langle S^2x_n, x_n\rangle}_{A}
= w_{A}(T^2) w_{A}(S^2).
\end{align*}
Again, by Theorem \ref{T.12}, we obtain
$w_{A}(T^2 + S^2) = w_{A}(T^2) + w_{A}(S^2)$, and hence
\begin{align*}
w_{A}(T^2 + S^2) = 2\max\big\{w^2_{A}(T), w^2_{A}(S)\big\}.
\end{align*}
\end{proof}
\textbf{Acknowledgement.}
Supported by a grant from Shanghai Municipal Science and Technology Commission (18590745200).
\bibliographystyle{amsplain}

\end{document}